\documentclass{amsart}

\newtheorem{theorem}{Theorem}[section]
\newtheorem{lemma}[theorem]{Lemma}
\newtheorem{cor}[theorem]{Corollary}

\theoremstyle{definition}
\newtheorem{definition}[theorem]{Definition}
\newtheorem{example}[theorem]{Example}

\newcommand{\fH}{{\mathfrak H}}
\newcommand{\fK}{{\mathfrak K}}
\newcommand{\fF}{{\mathfrak F}} 
\newcommand{\fN}{{\mathfrak N}}
\newcommand{\fS}{{\mathfrak S}}
\newcommand{\fX}{{\mathfrak X}}
\newcommand{\fY}{{\mathfrak Y}}
\newcommand{\id}{\triangleleft}
\newcommand{\str}{\ll}
\newcommand{\stains}{\gg}
\newcommand{\strG}{\ll_\text{G}}
\newcommand{\p}{{\mbox{$[p]$}}}
\newcommand{\scp}{{\mbox{$\scriptstyle [p]$}}}
\DeclareMathOperator{\ad}{ad}
\DeclareMathOperator{\soc}{Soc}
\DeclareMathOperator{\Hom}{Hom}
\DeclareMathOperator{\ch}{char}
\DeclareMathOperator{\loc}{Loc}
\newcommand{\cser}{\mathcal{C}}

\numberwithin{equation}{section}

\title[Strong containment]{Strong containment of saturated formations of soluble Lie algebras}
\author{Donald W. Barnes}
\address{1 Little Wonga Rd.\\Cremorne NSW 2090\\Australia\\}
\email{donwb@iprimus.com.au}

\subjclass[2010]{Primary 17B30, Secondary  20D10}
\keywords{Lie algebras, Schunck classes, saturated formations, local definition}

\begin{document}

\begin{abstract} It is shown that, if $\fH,\fK$ are saturated formations of soluble Lie algebras over a field of non-zero characteristic, and $\fH \stains \fK$ is a non-trivial example of strong containment, then $\fH = \fH/\fN$ and $\fH$ is not locally defined.
\end{abstract}

\maketitle
\section{Introduction}
The concept of strong containment for Schunck classes of finite soluble groups was introduced by Cline \cite{cline} in 1969.  It is discussed extensively in Doerk and Hawkes \cite{DH}.
\begin{definition} Let $\fH,\fK$ be Schunck classes of finite soluble groups.  We say that $\fH$ strongly contains $\fK$, written $\fH \stains \fK$ if, for every finite soluble group $G$, every $\fK$-projector of an $\fH$-projector of $G$ is a $\fK$-projector of $G$.
\end{definition}

Much attention is given to the special case where $\fH$ and $\fK$ are formations.  It is easy to give examples of strong containment of saturated formations of finite soluble groups.  If $\pi_1 \subset \pi_2$ are sets of primes, then the class of soluble $\pi_1$-groups is strongly contained in the class of soluble $\pi_2$-groups since a Hall $\pi_1$-subgroup of a Hall $\pi_2$-subgroup of $G$ is clearly a Hall $\pi_1$-subgroup of $G$.

I. S.  Guti\'errez Garc\'ia has asked in a private communication to the author, if there exist non-trivial examples of strong containment of saturated formations of soluble Lie algebras.

In the following, all Lie algebras are soluble and finite-dimensional over the field $F$ and $\fH, \fK$ are Schunck classes of soluble Lie algebras over $F$. 

\begin{definition} We say that $\fH$ is strongly contains  $\fK$, written $\fH \stains \fK$, if, for every soluble Lie algebra $L$ and $\fH$-projector $H$ of $L$, every $\fK$-projector of $H$ is a
$\fK$-projector of $L$. \end{definition}  

There are clearly three trivial cases, $\fK = 0$, $\fK = \fH$ and $\fH = \fS$, the class of all soluble Lie algebras. 

The cases of $\ch(F) =0$ and $\ch(F) \ne 0$ are dramatically different.  

\begin{theorem}\label{char0}  Suppose $\ch F = 0$.  Let $\fH \supseteq \fK$ be Schunck classes.  
Then $\fH \stains \fK$. \end{theorem}

\begin{proof} Every soluble Lie algebra over $F$ is completely soluble, so $\fH \stains \fK$ by Barnes and Newell \cite[Theorem 3.7]{BN}.
\end{proof}

For $\ch(F) \ne 0$, it is easy to produce examples of strong containment of Schunck classes, but the existence of non-trivial examples where $\fH$ and $\fK$ are formations remains unanswered.

To avoid continual reference to trivial cases, it is always assumed in the following that $\fS \ne\fH \ne \fK\ne 0$.  The class of all nilpotent algebras is denoted by $\fN$ and $0$ is used to denote the zero element, the zero algebra and the class containing only the zero algebra according to context. As the case of characteristic $0$ has been settled, in the following, it is assumed that $\ch(F) = p \ne 0$.  The socle of the Lie algebra $L$ is denoted by $\soc(L)$ and the nil radical of $L$ is denoted by $N(L)$.  If $V$ is an $L$-module, $\cser_L(V)$ denotes the centraliser of $V$ in $L$, that is, the kernel of the representation of $L$ on $V$.  If $\fF$ is a formation, the $\fF$-residual of the algebra $L$ is denoted by $L_{\fF}$.  This is the smallest ideal $K$ of $L$ with $L/K \in \fF$.

That a result is the Lie algebra analogue of a result in Doerk and Hawkes \cite{DH} is indicated by (DH, Lemma $x$, p. $y$).  Proofs which are exact translations are omitted.

\section{Strong containment}
In this section, we investigate basic properties of strong containment of Schunck classes.

\begin{lemma}\label{versions}  Suppose $\fH \stains \fK$.  Then every $\fK$-projector of a soluble Lie algebra $L$ is contained in some $\fH$-projector of $L$.
\end{lemma}

\begin{proof}  Let $L$ be a soluble Lie algebra of
least possible dimension with a $\fK$-projector $K$ not contained in any $\fH$-projector of $L$.  Let $A$ be a
minimal ideal of $L$.  Then $K+A/A$ is contained in some $\fH$-projector $H^*/A$ of $L/A$.  If
$H^* < L$, then there exists an $\fH$-projector $H$ of $H^*$ which contains $K$.  But $H$ is an
$\fH$-projector of $L$ by \cite[Lemma 1.8]{BGH}.  Therefore $H^* = L$, and $A$ is
complemented in $L$ by an $\fH$-projector $H$.  If $B$ is a minimal ideal of $L$ contained in $H$,
then $L/B \in \fH$ contrary to $H$ being an $\fH$-projector.  Therefore $L$ is primitive and $H$ is
faithfully represented on $A$.  Let $K_1 = H \cap (K+A)$.  Since $H \simeq L/A$, $K_1$ is a
$\fK$-projector of $H$ and so also of $L$ since $\fK \str \fH$.  Thus both $K$ and $K_1$ are
$\fK$-projectors of $K+A = K_1+A$.  By \cite[Lemma 1.11]{BGH}, there exists $a \in A$ such
that $\alpha_a(K_1) = K$ where $\alpha_a:L \to L$ is the automorphism $1+\ad_a$.  Then $\alpha_a(H)$ is an $\fH$-projector of $L$ which contains $K$.
\end{proof} 

\begin{lemma}\label{lem-min}  Let $L$ be an algebra of least possible dimension with an
$\fH$-projector $H$ and a $\fK$-projector $K$ of $H$ which is not a $\fK$-projector of $L$.  Then
$L$ is primitive with $H$ complementing $\soc L$. \end{lemma}

\begin{proof} Clearly $H \not= L$. Let $A$ be a minimal ideal of $L$.  Suppose $H+A < L$.  Then
$H$ is an $\fH$-projector of $H+A$, so $K$ is a $\fK$-projector of $H+A$.  Also, $H+A/A$ is an
$\fH$-projector of $L/A$ and $K+A/A$ is a $\fK$-projector of $H+A/A$, so $K+A/A$ is a
$\fK$-projector of $L/A$.  As $K$ is a $\fK$-projector of $K+A/A$, it is a $\fK$-projector of $L$. 
Therefore $H+A = L$.  As this holds for every minimal ideal, there is only one minimal ideal and
$L$ is primitive. \end{proof}

\begin{definition} The {\em boundary} of the Schunck class $\fX$ is the class $ b(\fX)$ of those Lie algebras not in $\fX$ whose proper quotients are in $\fX$.
A class $\fY$ of primitive algebras is called a {\em boundary class} if the intersection of $\fY$ with the class of proper quotients of algebras in $\fY$ is empty. 
\end{definition}
Clearly, $b(\fX)$ is a boundary class.  That every boundary class is the boundary of a Schunck class follows as in (DH, 2.3, p. 284).

\begin{definition}[DH 4.15, p. 308] The {\em avoidance class} of $\fH$ is the class $a(\fH)$ of primitive algebras $P$  with $H \cap \soc(P) = 0 $ for all $\fH$-projectors $H$ of $P$.  \end{definition}

Clearly, $b(\fH) \subseteq a(\fH)$.

\begin{lemma}  Let $P \in a(\fK)$ and let $M$ complement $A = \soc(P)$.  Then $M$ contains an $\fK$-projector of $P$. 
\end{lemma}

\begin{proof}  Let $U$ be a $\fK$-projector of $P$.  Then $U \cap A = 0$.  Let $B/C$ be a composition factor of $A$ as $U$-module.  Then $U+C/C$ is an $\fK$-projector of $U+B/C$ and $H^1(U, B/C) = 0$.  Thus $H^1(U, A) = 0$ and so, if $V$ complements $A$ in $U+A$, then there exists $a \in A$ with
$\alpha_a(U) = V$ where $\alpha_a: P \to P$ is the automorphism $1+\ad_a$.  In particular, for $V = M \cap (U+A)$, we have that $V = \alpha_a(U)$ is a $\fK$-projector of $P$.
\end{proof}

\begin{theorem}[DH 1.5, p. 429] Let $\fH \supset \fK$ be Schunck classes.  Then $\fH \stains \fK$ if and only if $b(\fH) \subseteq a(\fK)$.
\end{theorem}
 
\begin{proof}  Suppose  $\fK \not\str \fH$.  Then by Lemma \ref{lem-min}, there exists a primitive
algebra $L \in b(\fH) \backslash a(\fK)$ and $b(\fH) \not\subseteq a(\fK)$.   Suppose  $\fK \str \fH$.   Let $P \in b(\fH)$, let $K$ be a $\fK$-projector of $P$.  Then $K \le H$ for some $\fK$-projector $H$ of $P$.  Since $H \cap\, \soc(P) = 0$, we have $K \cap\, \soc(P) = 0$ and $P \in a(\fK)$.   
\end{proof}

 There exist non-trivial examples of Schunck classes with $\fH \stains \fK$.

\begin{example}\label{ex-strb} Let $\fK$ be a Schunck class and suppose that $b(\fK)$ contains more
than one (isomorphism type of) primitive algebra.  Let $\fX$ be a non-empty subclass of $b(\fK)$, $\fX
\ne b(\fK)$.  Let $\fH$ be the Schunck class with boundary $\fX$.  Then $\fX \subset a(\fK)$ and we have $\fK \str \fH \ne \fS$.
\end{example}

\section{Formations}
We now investigate the special case in which the Schunck classes $\fH, \fK$ are formations.  Our investigation parallels the work of D'Arcy set out in Chapter VII of Doerk and Hawkes \cite{DH}.  D'Arcy uses the formation functions $f,g$ of the canonical local definitions of the saturated formations $\fK, \fH$ and obtains the following necessary and sufficient condition for $\fK \str \fH$.

\begin{theorem}[DH, VII.5.1, p. 509]\label{Darcy}Let $f,g$ be the canonical definitions of the saturated formations $\fK, \fH$ of finite soluble groups.  Then $\fK \str \fH$ if and only if, for each $H \in \fH$ and $\fK$-projector $K$ of $H$, we have $H_{g(p)} \subseteq K_{f(p)}$ for all $p \in \ch(\fF)$.
\end{theorem}

A locally defined formation of soluble Lie algebras has a single defining formation, not a family as in the group case, which simplifies our analysis, as does our only needing to find a necessary condition for $\fK \str \fH$.  It is complicated by the fact that   not all saturated formations of soluble Lie algebras are locally defined. We need a substitute for the defining formation.  This is provided by the quotient formation $\fH/\fN$ defined as follows.

\begin{definition}\label{def-div} Let $\fH$ be a saturated formation.  We  define
the quotient of $\fH$ by $\fN$ to be
$$\fH/\fN = \{L/A \mid L \in \fH, A \id L, N(L) \le A\}.$$ 
\end{definition}

By Barnes \cite[Lemma 3.2]{local}, $\fH/\fN$ is a formation.  If $\fH$ is locally defined by $\fF$, then $\fF = \fH/\fN$ by \cite[Theorem 3.3]{local}.

  Suppose that $V$ is an $L$-module, $K$ is an ideal of $L$ and that $K \in \fK$ for some saturated formation $\fK$.  By Barnes \cite[Lemma 1.1]{Ado}, there is an $L$-module direct decomposition $V = V^{(K,\fK+)}\oplus V^{(K, \fK-)}$ where, as $K$-modules,  $V^{(K,\fK+)}$ is $\fK$-hypercentral and $V^{(K, \fK-)}$ is $\fK$-hypereccentric. 

\begin{lemma} \label{lem-nilform}  Let $A \in \fK$ be a subalgebra of $L$ and let $B$ be a nilpotent ideal of $L$.  Suppose that $L = A+B$.  Suppose that  $V^{(A,\fK+)} \subseteq V^{(B,\fN+)} $ for every $L$-module $V$.  Then $A \supseteq B$.
\end{lemma}

\begin{proof} Let $L$ be a minimal counterexample and let $K$ be a minimal ideal of $L$.  Then $L/K, A+K/K, B+K/K$ satisfy the conditions of the lemma, so $A+K \supseteq B+K$.  Thus $A+K = L$.  The result holds if $A=L$, so $A$ contains no minimal ideal of $L$.  It follows that $L$ is primitive and $N(L)= K$, so $B=K$.  Now $L$ has a faithful, completely reducible $L$-module.  Since $L$ has only one minimal ideal, it follows that there exists a faithful irreducible $L$-module $V$.  But $B$ acts nilpotently on the $L$-submodule $V^{(B,\fN+)}$.  Since $V$ is faithful and irreducible, this implies that $V^{(B,\fN+)} = 0$.  Therefore $V^{(A,\fK+)} = 0$.

 Let $\eta: L \to A$ be the epimorphism $\eta(x) = (x+B) \cap A$.  Since $A \cap B = 0$, we can define a new action $x \cdot v = \eta(x)v$ of $L$ on $V$.   Let $W$ be $V$ with this new action.  Put $X = \Hom_F(W,V)$  and we have $X^{(B,\fN+)}=0$.  But for the identity function $f(v)= v$, we have $af = 0$ for all $a \in A$ and $\langle f \rangle$ is the trivial $A$-module. It is $\fK$-central, contrary to $X^{(A,\fK+)} \subseteq X^{(B,\fN+)} $.
\end{proof}

\begin{lemma} \label{nilformsub} Let $A \in \fK$ be a subalgebra of $L$ and let $B$ be a nilpotent ideal of $L$.   Suppose that for every $L$-module $V$, we have $V^{(A,\fK+)} \subseteq V^{(B,\fN+)} $.  Then $A \supseteq B$.
\end{lemma} 

\begin{proof}   Put $M = A+B$.  We prove that $M, A, B$ satisfy the conditions of the lemma.   Let $(L^e, \p)$ be a $p$-envelope of $L$ and let
$U$ be the universal \p-enveloping algebra of $L^e$.  Let $U_1 \subseteq U$ be the universal    
\p-enveloping algebra of the \p-closure $M_\scp$ of $M$.  Let $V$ be any $M$-module.  Let $W = U
\otimes_{U_1} V$ be the induced $L^e$-module.  Then $W$ is an $L$-module, so we have
$W^{(A,\fK+)} \subseteq W^{(B,\fN+)}$.  But $V_1 = \{1 \otimes v \mid v \in V \}$ is an $M$-submodule
isomorphic to $V$.  Since
$$V_1^{(A,\fK+)} = W^{(A,\fK+)} \cap V_1 \subseteq W^{(B,\fN+)} \cap V_1 = V_1^{(B,\fN+)},$$
we have $V^{(A,\fK+)}  \subseteq V^{(B,\fN+)}$ for every $M$-module $V$.  By Lemma \ref{lem-nilform}, $A \supseteq B$. 
\end{proof}

\begin{lemma} \label{divbyN} Let $\fH$ be a saturated formation and let $\fF = \fH/\fN$.  Let $H \in \fH$ and let $V$ be an $H$-module.  Then $V^{(H,\fH+)} \subseteq V^{(H_{\fF}, \fN+)}$.  
\end{lemma}

\begin{proof} Consider first the case where $V$ is an $\fH$-central irreducible $H$-module.  Let $C$ be the centraliser of $V$ in $H$.  Since $V$ is $\fH$-central, the split extension $X$ of $V$ by $H/C$ is in $\fH$ and $X/V \in \fF$.  Thus $C \supseteq H_{\fF}$.  From this, it follows for any $V$, that $H_{\fF}$ acts nilpotently on $V^{(H,\fH+)}$.  Thus $V^{(H,\fH+)} \subseteq V^{(H_{\fF}, \fN+)}$.
\end{proof}

Now for the Lie algebra analogue of Theorem \ref{Darcy}

\begin{theorem} \label{KstrH}  Suppose $\fH  \stains \fK$ are saturated formations.  Let $\fF = \fH/\fN$.   Then  for each $H \in \fH$,  the $\fF$-residual $H_\fF$ is $\fK$-hypercentral.  
\end{theorem}

\begin{proof}     Let $H \in \fH$ and let $K$ be a $\fK$-projector
of $H$.  Let $V$ be an $H$-module and let $L$ be the split extension of $V$ by $H$.  Put $W = V^{(H,\fH+)}$. By Lemma \ref{divbyN},  $W \subseteq V^{(H_{\fF}, \fN+)}$.
Now $W+H$ is the unique $\fH$-projector of $L$ which contains $H$.  Also,  $X = V^{(K,\fK+)}$ is a $K$-submodule and $X+K$ is the unique $\fK$-projector
of $L$ which contains $K$, while $(X\cap W)+K$ is the unique $\fK$-projector of $W+H$ which
contains $K$.  Since $\fK \str \fH$, we must have $X \subseteq W$.  By  Lemma \ref{nilformsub}, $H_\fF\subseteq K$ and by Barnes  \cite[Theorem 4]{Cov}, $H_{\fF}$ is $\fK$-hypercentral.
\end{proof}

\section{Guti\'errez Garc\'ia containment} 
In \cite{Gstr}, I. Guti\'errez Garc\'ia introduced two weakened versions of strong containment for finite soluble groups, called G- and D-strong containment.  

\begin{definition} Let $\fF$ and $\fH$ be two saturated formations of finite soluble groups with $h$ the canonical local definition of $\fH$.  Suppose $\ch(\fF) \subseteq \ch(\fH)$.  We say that $\fF$ is G-strongly contained in $\fH$, written $\fF \strG \fH$, if, for each $H \in \fH$, an $\fF$-projector $E$ of $H$ satisfies $H_{h(p)} \subseteq E$ for each $p \in \ch(\fF)$. 

We say that $\fF$ is D-strongly contained in $\fH$, written $\fF \str_\text{D} \fH$, if, for each $H \in \fH$ an $\fF$-projector $E$ of $H$ satisfies $H_{h(p)} \subseteq E$ for each $p \in \ch(\fH)$.
\end{definition}

For Lie algebras, there are no considerations of different primes and we can avoid the assumption that the formations are locally defined. 

\begin{definition}  Let $\fK \subset\fH$ be saturated formations of soluble Lie algebras.  Let $\fF = \fH/\fN$. We say that $\fK$ is Guti\'errez Garc\'ia contained in $\fH$, written $\fK \strG \fH$, if for all $H \in \fH$,  $H_\fF$ is $\fK$-hypercentral.
\end{definition}

This is equivalent to the condition that for each $H \in \fH$, there exists a $\fK$-projector $K$ of $H$ such that $H_{\fF} \subseteq K$.    By Theorem \ref{KstrH}, if $\fK \str \fH$ then $\fK \strG \fH$.

In the following, $\fK \strG \fH$ and $\fF = \fH/\fN$.

\begin{lemma}\label{Uec} Suppose $\fH \ne \fF, \fS$.  Then there exists $L \notin\fF$ with a minimal ideal $A$ such that $L/A \in \fF$ and with an $\fH$-central but $\fK$-eccentric module $U$ such that $L/\cser_L(U) \in \fK$.
\end{lemma}

\begin{proof}   Take $L^1 \in \fH, L^1 \notin \fF$ of least possible dimension.  Let $A$ be a minimal ideal of $L^1$.  Take $P \in \fH, P \notin \fK$ of least possible dimension.  Let $U = \soc(P)$ and $Q = P/U$.  Then $Q \in \fF \cap \fK$ and $U$ is an $\fH$-central and $\fK$-eccentric $Q$-module.  Put $L = L^1 \oplus Q$.  Then $A$ is a minimal ideal of $L$ and $L/A \in \fF$ and $U$ is an $L$-module with the required properties.
\end{proof}

\begin{lemma}\label{nontriv}  Let $A$ be an ideal of the Lie algebra $L$ and let $V$ be an $L$-module with $AV \ne 0$.  Then there exists a section $V' = X/Y$ of $V$ such that $AV'$ is the only minimal submodule of $V'$.
\end{lemma}

\begin{proof} Take $X$ a submodule of $V$ of least possible dimension subject to the requirement that $AX \ne 0$.  Take $Y \subset X$ of largest possible dimension subject to $Y \not\supseteq AX$.  Then $V' = X/Y$ has the required properties.
\end{proof}

\begin{lemma}\label{indec}  Let $L,A,U$ be as given by Lemma \ref{Uec}.  There exists an $\fH$-hypercentral  $L$-module $V$ with the following properties:
\begin{description}
\item[(a)] $V$ has a unique minimal submodule $W$.
\item[(b)] $AV=W$.
\item[(c)] $W$ is $\fK$-eccentric.
\end{description}
\end{lemma}

\begin{proof}  By Barnes \cite[Theorem 5.1]{Ado}, there exists a faithful $\fH$-hypercentral $L$-module $V_1$.  We have $AV_1 \ne 0$.  By Lemma \ref{nontriv}, there exists a section $V_2$ of $V_1$ with $W_2 = AV_2$ satisfying the conditions (a) and (b).  If $W_2$ is $\fK$-eccentric, we are done, so suppose that $W_2$ is $\fK$-central.

Let $C = \cser_L(A)$.  Since $A$ is $\fK$-central, $L/C \in \fK$.  Also $L/\cser_L(U) \in \fK$.  Put $D = C \cap \cser_L(U)$.  Then $L/D \in \fK$ and $A, U$ are $L/D$-modules. It follows from Barnes \cite[Theorem 2.3]{B6}, that $U \otimes W_2$ is $\fK$-hypereccentric.  We form $V_3 = U \otimes V_2$.  Then $V_3$ is an $\fH$-hypercentral $L$-module by Barnes \cite[Theorem 2.1]{B2}.  Further, we have $AV_3 = U \otimes W_2 \ne 0$.  By Lemma \ref{nontriv},  we obtain a section $V$ which, with $W = AV$, satisfies all the conditions (a), (b) and (c).
\end{proof}

\begin{theorem} \label{noG} Suppose that $\fK \strG \fH$ and that $0 \ne \fK \ne \fH \ne \fS$. Then $\fH = \fH/\fN$.
\end{theorem}

\begin{proof}  Let $\fF = \fH/\fN$.   Take $L, A, U, V, W$ as given by  Lemmas  \ref{Uec} and \ref{indec}.  Let $L^*$ be the split extension of $V$ by $L$.  Then $L^* \in \fH$ since $L \in \fH$ and $V$ is $\fH$-hypercentral.

Let $X$ be any non-zero ideal of $L^*$ which is contained in $A+V$.  If $X \subseteq V$, then $X \supseteq W$ since $W$ is the only minimal submodule of $V$.  If $X \not\subseteq V$, then there exists $a+v \in X$, $a\in A, v \in V$ with $a \ne 0$.  The centraliser $\cser_A(V)$ is an ideal of $L$ and, as $A$ is minimal and acts non-trivially, $\cser_A(V)=0$.  Thus $aV \ne 0$.  Since $X$ is an ideal of $L^*$, $aV \subseteq X \cap W$.  It follows that $X \supseteq W$. 

Consider $L^*_{\fF}$.  As $L_{\fF} = A$, we have $L^*/(A+V) \in \fF$.  Thus $L^*_{\fF} \subseteq A+V$.  If $L^*_{\fF} \ne 0$, then  $L^*_{\fF} \supseteq W$.  But $W$ is $\fK$-eccentric contrary to $\fK \strG \fH$.  Therefore $L^*_{\fF}=0$ and $L^* \in \fF$.  But $L \simeq L^*/(V+M)$, so $L \in \fF$ contrary to the choice of $L$.  Therefore $\fH = \fF$.
\end{proof}

\begin{lemma}\label{notequal} Suppose $\loc(\fF) = \fH \ne \fS$.  Then $\fF \ne \fH$.
\end{lemma}

\begin{proof} Suppose $\fF = \fH$.  Take $L \notin \fH$ of least possible dimension.  Then $L$ is primitive.  Let $A = \soc(L)$.  Then $L/A \in \fH = \fF$, so $L \in \loc(\fF) = \fH$ contrary to assumption.
\end{proof}

\begin{cor}  Suppose that $\fK \str \fH$ non-trivially.  Then $\fH = \fH/\fN$ and $\fH$ is not locally defined.
\end{cor}
\begin{proof}  By Theorem \ref{KstrH}, $\fK \strG \fH$. By Lemma \ref{notequal}, $\fH$ is not locally defined.
\end{proof}

\bibliographystyle{amsplain}

\end{document}